\documentclass{amsart}
\usepackage{mathtools, microtype, mathrsfs, xfrac, hyperref, amssymb, enumerate, xspace, stmaryrd}
\usepackage[utf8]{inputenc}
\usepackage{tikz-cd}
\usepackage{pdflscape}
\usepackage{bm}

\usepackage{todonotes}

\numberwithin{equation}{section}

\numberwithin{subsection}{section}

\allowdisplaybreaks[1]

\newtheorem*{namedtheorem}{\theoremname}
\newcommand{\theoremname}{testing}

\newtheorem{theorem}{Theorem}
\newtheorem{proposition}[theorem]{Proposition}
\newtheorem{proposition-definition}[theorem]
{Proposition-Definition}

\newtheorem{lemma}[theorem]{Lemma}
\newtheorem*{theorem*}{Theorem}

\theoremstyle{definition}

\newtheorem{example}[theorem]{Example}

\newtheorem{remark}[theorem]{Remark}

\newtheorem*{question*}{Question}

\theoremstyle{remark}

\renewcommand{\mathcal}{\mathscr}

 \newcommand\cB{\mathcal{B}}

\newcommand\cM{\mathcal{M}}

\newcommand\CC{\mathbb{C}} 
 \newcommand\FF{\mathbb{F}}

 \newcommand\PP{\mathbb{P}}
\newcommand\QQ{\mathbb{Q}} \newcommand\RR{\mathbb{R}}

 \newcommand\ZZ{\mathbb{Z}}

\newcommand\fC{\mathfrak{C}}

\newcommand\arr{\ifinner\to\else\longrightarrow\fi}

\newcommand\arrto{\ifinner\mapsto\else\longmapsto\fi}

\renewcommand\H{\operatorname{H}}

\newcommand{\eqdef}{\mathrel{\smash{\overset{\mathrm{\scriptscriptstyle def}} =}}}

\def\displaytimes_#1{\mathrel{\mathop{\times}\limits_{#1}}}

\def\displayotimes_#1{\mathrel{\mathop{\bigotimes}\limits_{#1}}}

\newcommand\aut{\operatorname{Aut}}

\newcommand\pic{\operatorname{Pic}}

\newcommand\spec{\operatorname{Spec}}

\newcommand{\underaut}{\mathop{\underline{\mathrm{Aut}}}\nolimits}

\newlength{\ignora}

\newcommand{\gal}{\operatorname{Gal}}

\DeclareFontFamily{U}{mathx}{\hyphenchar\font45}
\DeclareFontShape{U}{mathx}{m}{n}{
      <5> <6> <7> <8> <9> <10>
      <10.95> <12> <14.4> <17.28> <20.74> <24.88>
      mathx10
      }{}
\DeclareSymbolFont{mathx}{U}{mathx}{m}{n}
\DeclareFontSubstitution{U}{mathx}{m}{n}
\DeclareMathAccent{\widecheck}{0}{mathx}{"71}
\DeclareMathAccent{\wideparen}{0}{mathx}{"75}

\renewcommand{\epsilon}{\varepsilon}

\newcommand{\upic}{\underline{\pic}}

\DeclareMathOperator{\lcm}{lcm}

\begin{document}

\title{On Grothendieck's section conjecture for curves of index $1$}

\author{Giulio Bresciani}

\begin{abstract}
	We prove that every hyperbolic curve with a faithful action of a non-cyclic $p$-group (with a few exceptions if $p=2$) has a twisted form of index $1$ which satisfies Grothendieck's section conjecture. Furthermore, we prove that for every hyperbolic curve $S$ over a field $k$ finitely generated over $\QQ$ there exists a finite extension $K/k$ and a finite étale cover $C\to S_{K}$ such that $C$ satisfies the conjecture.
\end{abstract}

\address{Scuola Normale Superiore\\Piazza dei Cavalieri 7\\
56126 Pisa\\ Italy}
\email{giulio.bresciani@gmail.com}


\maketitle
Curves are smooth, projective and geometrically connected over a base field $k$.

Given a hyperbolic curve $C$ over a field $k$ finitely generated over $\QQ$, there is a short exact sequence of étale fundamental groups
\[1\to\pi_{1}(C_{\bar{k}})\to\pi_{1}(C)\to\gal(\bar{k}/k)\to 1.\]
Denote by $\Pi_{C/k}(k)$ the set of sections of this short exact sequence modulo the action by conjugation of $\pi_{1}(C_{\bar{k}})$, there is a natural map
\[C(k)\to\Pi_{C/k}(k).\]

In 1983, in a famous letter to Faltings \cite{gro97}, Grothendieck conjectured that this map is bijective; this statement is known as the \emph{section conjecture}. Injectivity was already known to Grothendieck and follows from the Mordell-Weil theorem; the hard part is surjectivity. 
Some cases of the conjecture are known, let us briefly recall them.

\begin{itemize}
	\item The conjecture holds if $k$ has an embedding $k\subset\RR$ such that $C(\RR)=\emptyset$. The literature contains several proofs starting with Huisman in 2001 \cite[Proposition 4.2 (7.)]{huisman} \cite[Corollary 3.13]{mochizuki} \cite[Appendix A]{sti10} \cite{pal} \cite{wickelgren}. The result was probably known even before Huisman, though; exact attribution is unclear.
	\item In 2008, J. Stix proved it in the case in which $k=\QQ$ and $C$ maps to a non-trivial Brauer-Severi variety \cite{sti10} (see \cite{giulio-qnf} for a generalization).
	\item Later in 2008, D. Harari and T. Szamuely proved it in the case in which the cohomology class of $\upic^{1}_{C}$ in $\H^{1}(k,\upic^{0}_{C})$ is not in the maximal divisible subgroup \cite{harari-szamuely}.
	\item In 2010, R. Hain proved it for the generic fiber of the universal curve over $\cM_{g}$, $g\ge 5$, over a field finitely generated over $\QQ$ \cite{hain}.
\end{itemize}

Recall that the index of a curve $C$ is the greatest common divisor of the degrees of the residue fields of the closed points of $C$. All the results above share a common feature: they are all about curves of index larger than $1$. This is obvious for curves without real points and for the results of Stix and Harari-Szamuely; it is the so-called Franchetta conjecture -- proved by Beauville (unpublished) and Arbarello-Cornalba \cite{arbarello-cornalba} -- for the result of Hain. While A. Tamagawa showed that it is sufficient to prove the conjecture for curves without rational points \cite[Corollary 101]{sti13}, we still need to work with curves of index $1$. 

It is not by chance that all the available results are about curves of index larger than $1$: at least in the first three cases listed above, it can be checked that the authors prove that the \emph{abelianized} fundamental group (that is, the étale fundamental group of $\upic^{1}_{C}$) has no Galois sections \cite[Chapter 3]{sti13}, even if they do not state it explicitly. If a curve has index $1$, then $\upic^{1}_{C}(k)\neq\emptyset$ and hence the abelianized fundamental group has a Galois section.

In order to study the conjecture for curves of index $1$, we need to rely more on the anabelian nature of the fundamental group.

\subsection*{Curves of index 1} If we fix any finite group $G$, the conjecture for all curves with a free action of $G$ over all finitely generated extensions of $\QQ$ implies the full conjecture; this is an easy consequence of well known facts \cite[Corollary 107, Proposition 111]{sti13}. Because of this, working with curves that have many automorphisms is not restrictive.

We prove that every hyperbolic curve with a faithful action of a non-cyclic $p$-group (with a few exceptions if $p=2$) has a twisted form of index $1$ which satisfies the conjecture.

\begin{theorem}\label{theorem:index}
	Let $C$ be a hyperbolic curve over a field $k$ finitely generated over $\QQ$ with a faithful action of a finite, non-cyclic $p$-group $G$ for some prime $p$. 
	
	If $p=2$, $G$ has a cyclic subgroup of index $2$ and the action is not free, furthermore assume that $G$ is not dihedral and that the genus is odd.
	
	There exists a finitely generated extension $K/k$ and a twisted form $\fC$ of $C_{K}$ such that $\fC$ has index $1$ and $\Pi_{\fC/K}(K)=\emptyset$.
\end{theorem}

There is a birational version of the section conjecture which is arguably more approachable, and more is known about it \cite{koenigsmann} \cite{esnault-wittenberg} \cite{hoshi} \cite{sti15} \cite{saidi-tyler} \cite{giulio-tbir}. In particular, H. Esnault and O. Wittenberg have proved a result about curves of index $1$, let us describe it.

Let $C$ be a curve over a number field $k$. Denote by $G_{k(C)}^{\rm [ab]}$ the Galois group of $\bar{k}(C)^{\rm ab}/k(C)$, where $\bar{k}(C)^{\rm ab}$ is the maximal abelian subextension of $\overline{k(C)}/\bar{k}(C)$. There is a natural projection $G_{k(C)}^{\rm [ab]}\to\gal(\bar{k}/k)$; we might think of it as a birational, abelianized version of $\pi_{1}(C)\to\gal(\bar{k}/k)$. Assuming that the Tate-Shafarevich group of the Jacobian of $C$ is finite, H. Esnault and O. Wittenberg have proved that $C$ has index $1$ if and only if $G_{k(C)}^{\rm [ab]}\to\gal(\bar{k}/k)$ is split.

While at first glance Theorem \ref{theorem:index} might seem in conflict with this result, they are actually perfectly compatible. In fact, if we believe in Grothendieck's anabelian philosophy and $C$ is a hyperbolic curve of index $1$ with $C(k)=\emptyset$, then we should expect $\pi_{1}(C)\to\gal(\bar{k}/k)$ to be non-split even though the \emph{abelianized} fundamental group has a splitting induced by some point of $\upic^{1}_{C}(k)\neq\emptyset$ (ignoring the ``birational'' side of the story for clarity's sake).

\subsection*{Étale covers} 

Recall that, if $C\to S$ is a finite étale cover and the section conjecture holds for $C_{K}$ \emph{for all finite extensions $K/k$}, then the section conjecture holds for $S$ \cite[Proposition 111]{sti13}. Using the same methods as in Theorem~\ref{theorem:index}, we prove that every hyperbolic curve, after a finite extension of the base field, has a finite étale cover which satisfies the conjecture.

\begin{theorem}\label{theorem:cover}
	Let $S$ be a hyperbolic curve over a field $k$ and $\ell$ a prime. Assume either that $k$ is finitely generated over $\QQ$, or that it is a finite extension of $\QQ_{p}$ with $p\neq\ell$. There exists a finite extension $K/k$ and a finite étale $\left(\ZZ/{\ell}\right)^{2}$-cover $\fC\to S_{K}$ such that $\Pi_{\fC/K}(K)=\emptyset$.
\end{theorem}

Notice that in the statement of Theorem~\ref{theorem:cover} $K/k$ is a \emph{finite} extension, while in Theorem~\ref{theorem:index} $K/k$ is \emph{finitely generated}: we obtain a better control over $K/k$ in exchange for losing control over the index of $\fC$.

\subsection*{Acknowledgements}

I would like to thank A. Vistoli for many useful discussions: this paper would not exist without the frequent brainstorming sessions I had with him. I would also like to thank J. Stix for many useful comments which have substantially improved the exposition of the paper. Among other things, he corrected a minor mistake in the proof of Theorem~\ref{theorem:cover}, he told me about the isomorphism $\ZZ/{2}\ast\ZZ/{2}\simeq\ZZ\rtimes\ZZ/{2}$ (see Lemma~\ref{lemma:dihedral}), he provided an alternative proof of Proposition~\ref{proposition:obstruction} (see Remark~\ref{remark:altstix}), and a discussion with him inspired a much shorter proof using cohomological dimension of Proposition~\ref{proposition:subcyclic} in the case of curves (unfortunately, I was not able to adapt it to the case of orbicurves). Finally, I would like to thank H. Esnault for some useful remarks.

\subsection*{Étale fundamental gerbes}

We use the formalism of \emph{étale fundamental gerbes} introduced by N. Borne and A. Vistoli \cite[\S 8]{borne-vistoli} \cite[Appendix A]{giulio-anab}: it is an alternative point of view on the theory of étale fundamental groups which is particularly well-suited for studying the section conjecture. From this point of view, the main object is \emph{not} the fundamental group, but \emph{the space of Galois sections itself}, with a richer structure and constructed directly without passing through fundamental groups.

If $S$ is a geometrically connected scheme (or algebraic stack) over a field $k$, the étale fundamental gerbe of $S$ is a pro-finite étale gerbe $\Pi_{S/k}$ over $k$ with a structural morphism $S\to\Pi_{S/k}$ which is universal for this property: for every morphism $S\to\Phi$ where $\Phi$ is a pro-finite étale gerbe, there exists a factorization $S\to\Pi_{S/k}\to\Phi$ unique up to a unique isomorphism. 

The space of Galois sections of $S$ is in natural bijection with the set of isomorphism classes of the groupoid $\Pi_{S/k}(k)$. If $S$ is a hyperbolic curve over a field finitely generated over $\QQ$, then it is known that $\Pi_{S/k}(k)$ is already a set (as opposed to a groupoid), namely its elements have no non-trivial automorphisms \cite[Proposition 104]{sti13} \cite[Lemma 2.6]{giulio-anab}.

\section{An obstruction for Galois sections of twisted forms}\label{sect:obstruction}

We give an obstruction for the existence of Galois sections of twisted forms. We developed the obstruction for completely different purposes, namely proving the existence of certain cycles on $\PP^{2}$ not defined over their field of moduli \cite[Theorem 4]{giulio-fmod} \cite{giulio-points} \cite{giulio-p2}, but the technique is completely general and, as A. Vistoli pointed out to us, can be applied to the study of the section conjecture. As J. Stix pointed out to us, a similar obstruction has already been used by M. Stoll \cite[Lemma 5.5]{stoll} (see also \cite[Theorem 142]{sti13}).

Let $S$ be a geometrically connected algebraic stack of finite type over a field $k$. Our main example is the case in which $k$ is finitely generated over $\QQ$ and $S$ is a smooth, projective, hyperbolic orbicurve, but the obstruction is completely general.

We stress the fact that allowing $S$ to be an algebraic stack, rather than a scheme, is an important assumption. While we could work with schemes, we would obtain significantly weaker results. For instance, restricting ourselves to schemes would force us to assume that the action is free in Theorem~\ref{theorem:index}.

Consider a short exact sequence of group schemes
\[1\to N\to E\to G\to 1\]
pro-finite étale over $k$. Suppose that we have a geometrically connected $E$-torsor $U\to S$, i.e. an $E$-covering; denote by $C$ the quotient stack $[U/N]$, it is a $G$-covering.

\[\begin{tikzcd}
	U\ar[d,"N"]\ar[dd,bend right=45,swap,"E"]	\\
	C\ar[d,"G"]								\\
	S	
\end{tikzcd}\]

\begin{example}
	If $E$ is a \emph{constant} finite group scheme, i.e. just a finite group, then $U$ is simply a Galois $E$-covering, while $C=[U/N]$ is a Galois $G$-covering. 
\end{example}

\begin{example}
	Consider $s:\spec k\to\Pi_{S/k}$ a Galois section and define $U$ as the fibered product
	\[U\eqdef \spec k\times_{\Pi_{S/k}}S.\]
	The morphism $U\to S$ is the projective limit of all the étale neighbourhoods of $s$: it is a $k$-form of the universal covering of $S$, namely $U_{k^{s}}\to S$ is the universal covering, where $k^{s}/k$ is a separable closure. We have that $U$ is an $E=\underaut_{\Pi_{S/k}}(s)$-torsor where $E$ is the group scheme of automorphisms of $s\in\Pi_{S/k}(k)$ and $C=[U/N]$ is some étale neighbourhood of $s$. If $s$ is associated with a rational point $p\in S(k)$, then $E=\underline{\pi}_{1}(S,p)$ is the étale fundamental group scheme of $(S,p)$.
\end{example}

The $E$-torsor $U\to S$ induces morphisms
\[S\to\cB E\to\cB G,\]
where $\cB E$ and $\cB G$ are the classifying stacks of $E$ and $G$ respectively. Recall that the set of rational points modulo isomorphism $\cB G(k)/\sim$ of $\cB G$ is the set of $G$-torsors over $\spec k$ modulo isomorphism, which coincides with the non-abelian cohomology set $\H^{1}(k,G)$, and similarly for $E$.

Choose $T\to \spec k$ a $G$-torsor. We may use $T$ to define a twisted form $\fC$ of $C$, namely
\[\fC\eqdef C\times^{G}T=(C\times T)/G\]
where $G$ acts on both sides of $C\times T$ simultaneously. Equivalently, we may define $\fC$ as the fibered product
\[\begin{tikzcd}
	\fC\rar\dar		&	\spec k\dar		\\
	S\rar			&	\cB G,
\end{tikzcd}\]
where the map $\spec k\to\cB G$ is the one given by $T$.

\begin{proposition}\label{proposition:obstruction}
	If $T$ does not lift to $E$, i.e. its class is not in the image of $\H^{1}(k,E)\to\H^{1}(k,G)$, then $\Pi_{\fC/k}(k)=\emptyset$.
\end{proposition}

The underlying idea is the following (we are going to give a slightly different proof in order to avoid using the theory of group actions on stacks). We have a morphism $C\to\cB N$ associated with the covering $U\to C$, this induces a morphism between twists $\fC=C\times^{G}T\to\cB N\times^{G}T$. By definition of quotient stack, $\cB N\times^{G}T=[\cB N\times T/G]$ is the so-called \emph{gerbe of liftings} of $T$ to $E$, which has no rational sections since $T$ does not lift to $E$. By the universal property of the fundamental gerbe, there is a splitting $\fC\to\Pi_{\fC/k}\to\cB N\times^{G}T$, hence $\Pi_{\fC/k}(k)$ is empty as well.

\begin{proof}
	Consider the fibered product
	\[\begin{tikzcd}
		\Phi_{T}\rar\dar	&	\spec k\dar["T"]		\\
		\cB E\rar			&	\cB G
	\end{tikzcd}\]
	where $\Phi_{T}$ is the gerbe of liftings of $T$ to $E$. Since the two compositions $\fC\to\spec k\to\cB G$ and $\fC\to S\to\cB E\to\cB G$ are equivalent by construction, we get an induced map $\fC\to\Phi_{T}=\spec k\times_{\cB G}\cB E$.
	
	Since $\cB E$, $\cB G$ are pro-finite étale gerbes and $E\to G$ is surjective, then $\Phi_{T}$ is a pro-finite étale gerbe over $\spec k$ as well, and hence we have an induced map $\Pi_{\fC/k}\to\Phi_{T}$ by the universal property of the étale fundamental gerbe. By construction, $\Phi_{T}(k)/\sim$ is the set of $E$-torsors over $k$ which lift $T$, by hypothesis it is empty. It follows that $\Pi_{\fC/k}(k)$ is empty as well.
\end{proof}

\begin{remark}\label{remark:altstix}
	Let us sketch an alternative proof of Proposition~\ref{proposition:obstruction} in terms of étale fundamental groups, for which I thank J. Stix. Let $k^{s}$ be a separable closure of $k$ and $\Gamma=\gal(k^{s}/k)$. The torsors $U\to C\to S$ and $T\to \spec k$ correspond to homomorphisms 
	\[\pi_{1}(S)\to E(k^{s})\rtimes\Gamma\to G(k^{s})\rtimes \Gamma,\]
	\[\Gamma\to G(k^{s})\rtimes\Gamma\]
	over $\Gamma$, and $\pi_{1}(\fC)$ identifies with the fibered product of $\pi_{1}(S)$ and $\Gamma$ over $G(k^{s})\rtimes\Gamma$. A Galois section $\Gamma\to\pi_{1}(\fC)$ induces by composition a homomorphism $\Gamma\to\pi_{1}(\fC)\to\pi_{1}(S)\to E(k^{s})\rtimes\Gamma$. Since $\pi_{1}(\fC)$ is a fibered product over $G(k^{s})\rtimes\Gamma$, this gives a lifting $\Gamma\to E(k^{s})\rtimes\Gamma$ of the homomorphism $\Gamma\to G(k^{s})\rtimes\Gamma$ corresponding to $T$, which is absurd.
\end{remark}

\begin{remark}\label{remark:fundlift}
	In the particular case in which $U$ is the pro-finite limit of étale neighbourhoods of a Galois section of $S$, and hence $\Pi_{C/k}=\cB N$, it is not difficult to see that the gerbe of liftings $\Phi_{T}\simeq \cB N\times^{G}T=\Pi_{C/k}\times^{G}T$ coincides with the étale fundamental gerbe of the twist $\fC$.
\end{remark}

\section{About the étale fundamental group of orbicurves}

Let $S$ be a complete orbicurve over $\CC$ with fundamental group $\pi$. The aim of this section is to prove that, if $S$ is hyperbolic, a pro-$p$ group of finite derived length with a closed embedding in $\hat{\pi}$ is pro-cyclic (with one exception if $p=2$), see Proposition~\ref{proposition:subcyclic}. While this is an essential tool for us, it is only used at the very last step of the proof of Theorem~\ref{theorem:index}, and its pourpose might seem not clear before reaching that point. Furthermore, the proof of Proposition~\ref{proposition:subcyclic} is quite long. Because of this, depending on taste the reader might want to read the proof of Theorem~\ref{theorem:index} in Section~\ref{section:index} first.

We work over $\CC$ for simplicity. Since the étale fundamental group is invariant under base change of algebraically closed fields, the results hold over any algebraically closed field of characteristic $0$.

Recall that the rational Euler characteristic of $S$ is defined as 
\[\chi(S)=2-2g-\sum_{s}\frac{\gamma_{s}-1}{\gamma_{s}},\]
where the sum runs over non-trivial orbifold points $s$ of degree $\gamma_{s}$, namely $S$ is locally isomorphic to $[\CC/\mu_{\gamma_{s}}]$ around $s$. The orbicurve $S$ is hyperbolic if $\chi(S)<0$. The following is well known to experts, but we couldn't find a reference.

\begin{lemma}
	If $f:S'\to S$ is a finite étale cover of degree $d$ of complete orbicurves, then $\chi(S')=d\chi(S)$.
\end{lemma}

\begin{proof}
	Denote by $g,g'$ the genus of $S,S'$ respectively. If $s'\in S'$, denote by $e_{s'}$ the ramification index at $s'$ of the underlying morphism of coarse moduli spaces. By Riemann-Hurwitz, we get
	\[2g'-2=d(2g-2)+\sum_{s'\in S'}(e_{s'}-1),\]
	where the sum is finite since $e_{s'}=1$ if $s'$ is not a ramification point. Hence,
	\[\chi(S')=d(2-2g)-\sum_{s'\in S'}e_{s'}-1+\frac{\gamma_{s'}-1}{\gamma_{s'}}.\]
	Again, the sum is finite since $\gamma_{s'}=1$ if $s'$ is not an orbifold point. Since $S'\to S$ is étale, then $\gamma_{f(s')}=e_{s'}\gamma_{s'}$. We thus get
	\[\chi(S')=d(2g-2)+\sum_{s'\in S'}\frac{\gamma_{f(s')}-\gamma_{s'}+\gamma_{s'}-1}{\gamma_{s'}}=\]
	\[=d(2g-2)+\sum_{s'\in S'}e_{s'}\frac{\gamma_{f(s')}-1}{\gamma_{f(s')}}=d(2g-2)+d\sum_{s\in S}\frac{\gamma_{s}-1}{\gamma_{s}}\]
	where the last passage follows from the equality $\sum_{f(s')=s}e_{s'}=d$.
	
\end{proof}

Consider the action by $-1$ of $\ZZ/{2}$ on $\ZZ_{2}$, and write $D_{2^{\infty}}=\ZZ_{2}\rtimes \ZZ/{2}$ for the infinite, $2$-adic dihedral group. Here is the main result of this section.

\begin{proposition}\label{proposition:subcyclic}
	Assume that $S$ is hyperbolic, and let $G\subset\hat{\pi}$ be a closed subgroup which is a $p$-group for some prime $p$.
	
	If $G$ has finite derived length, it is either pro-cyclic or isomorphic to $D_{2^{\infty}}$. Furthermore, if $G\simeq D_{2^{\infty}}$, then $S$ has at least one orbifold point of even degree.
\end{proposition}

In order to prove Proposition~\ref{proposition:subcyclic}, we need some technical lemmas. 

\begin{lemma}\label{lemma:psub}
	A homomorphism $G\to H$ of pro-$p$-groups is surjective if and only if $G\to H^{\rm ab}/p$ is surjective, where $H^{\rm ab}$ is the abelianization.
\end{lemma}

\begin{proof}
	This is a direct consequence of Burnside's basis theorem.
\end{proof}

\begin{lemma}\label{lemma:free}
	Let $p$ be a prime number and
	\[\pi=\left<a_{1},b_{1},\dots,a_{g},b_{g}\mid \prod [a_{i},b_{i}]=1 \right>\]
	the fundamental group of a complete curve of genus $g\ge 2$. Let $r$ be either $1$ or $2$, and $H\subset \pi^{\rm ab}/p$ a subgroup of rank $\ge r$.

	There exists a surjective homomorphism 
	\[q:\pi\to Q=\ZZ^{\ast r}\]
	such that the composition $H\subset\pi^{\rm ab}/p\to Q^{\rm ab}/p\simeq \left(\ZZ/{p}\right)^{r}$ is surjective.
\end{lemma}

\begin{proof}
	For ease of notation, rename the generators $a_{1},b_{1},\dots,a_{g},b_{g}$ as $x_{1},\dots,x_{2g}$ respectively.
	
	If $r=1$ it is enough to define $\pi\to Q=\ZZ$ by mapping one suitably chosen generator $x_{i}$ to $1\in\ZZ$ and all the others to $0$.
	
	Assume $r=2$. Since $H\subset \pi^{\rm ab}/p=\FF_{p}x_{1}\oplus\dots\oplus\FF_{p}x_{2g}$ has rank $\ge 2$, we may choose two linearly independent vectors $(u_{i})_{i},(v_{i})_{i}\in H$. 
	
	Assume first that there are two indexes $1\le j<j'\le 2g$ which are \emph{not} of the form $j=2m-1$, $j'=2m$ and such that the $2\times 2$ matrix $\begin{psmallmatrix}u_{j} & v_{j} \\ u_{j'} & v_{j'}\end{psmallmatrix}$ has rank $2$. We may then define $\pi\to Q\simeq\ZZ\ast\ZZ$ by
	\[Q=\left<x_{1},\dots,x_{2g}\mid x_{i}=1\text{ if }i\neq j,j'\right>.\]
	
	If such indexes do not exist, then there exists $1\le m\le 2g$ such that $\begin{psmallmatrix}u_{2m-1} & v_{2m-1} \\ u_{2m} & v_{2m}\end{psmallmatrix}$ has rank $2$ \emph{and} for every $i\neq 2m-1,2m$ we have $u_{i}=v_{i}=0$, since $(u_{i},v_{i})$ must be a multiple of both $(u_{2m-1},v_{2m-1})$ and $(u_{2m},v_{2m})$. If $m\le g-1$, define
	\[Q=\left<x_{1},\dots,x_{2g}\mid x_{2m+1}=x_{2m},~x_{2m+2}=x_{2m-1},\right.\]
	\[\left.x_{i}=1\text{ if }i\neq 2m-1,2m,2m+1,2m+2\right>=\ZZ\ast\ZZ.\]
	The case $m=2g$ is analogous.
\end{proof}

\begin{lemma}\label{lemma:genus}
	Let $S$ be a complete hyperbolic orbicurve with fundamental group $\pi$ and $G\subset \hat{\pi}$ an infinite index closed subgroup such that $G^{\rm ab}$ has finite topological rank. For every $N$, there exists a finite index subgroup $\pi'\subset \pi$ such that $\hat{\pi}'\subset\hat{\pi}$ contains $G$ and such that the associated covering $S'\to S$ has genus larger than $N$.
\end{lemma}

\begin{proof}
	Let $r$ be the rank of $G^{\rm ab}$ and $m$ the least common multiple of the degrees of the orbifold points of $S$, if we pass to étale coverings of $S$ then $m$ can only decrease.
	
	Since $G$ has infinite index and $\chi(S)<0$, up to passing to a finite index subgroup of $\pi$ we might assume that $-\chi(S)$ is arbitrarily large. If by contradiction the genus of the coverings whose fundamental group contains $G$ is bounded by $N$, we might then assume that $S$ has an arbitrarily large number of non-trivial orbifold points, say at least $2N+4rm^{2}$.
	
	The torsion subgroup of $\hat{\pi}^{\rm ab}$ is $\left(\prod_{s}\ZZ/{\gamma_{s}}\right)/(1,\dots,1)$, where $s$ runs through all the orbifold points and $\gamma_{s}$ is the degree of $s$. If $p\in S$ is an orbifold point, let $H_{p}\subset\hat{\pi}^{\rm ab}$ be the subgroup generated by a loop around $p$, i.e. the image of $\ZZ/{\gamma_{p}}\subset\prod_{s}\ZZ/{\gamma_{s}}$. Since $(1,\dots,1)\in\prod_{s}\ZZ/{\gamma_{s}}$ has order $m$, for every non-trivial torsion element $g\in \hat{\pi}^{\rm ab}$ there are at most $m$ orbifold points $s$ such that $g\in H_{s}$. Moreover, the intersection of the image of $G^{\rm ab}$ and of the torsion subgroup of $\hat{\pi}^{\rm ab}$ is bounded by $rm$. This implies that the number of orbifold points $s\in S$ such that $H_{s}$ intersects non-trivially the image of $G^{\rm ab}\to\hat{\pi}^{\rm ab}$ is bounded by $rm^{2}$. We call these bad points, while the others are good points.
	
	Choose $n\le m$ good points $s_{1},\dots,s_{n}$ such that $\lcm(\gamma_{s_{1}},\dots,\gamma_{s_{n}})=\lcm_{s\text{ good}}(\gamma_{s})$. Consider the standard presentation $\pi=\left<a_{i},b_{i},x_{s}\mid x_{s}^{\gamma_{s}}=1, \prod_{i}[a_{i},b_{i}]\cdot\prod_{s}x_{s}=1\right>$. Define a homomorphism $f:\pi\to \ZZ/{m}$ as follows. All the generators $a_{i},b_{i}$ map to $0$. If $s$ is a bad point, $x_{s}$ maps to $0$. If $s$ is a good point different from $s_{1},\dots,s_{n}$, $x_{s}$ maps to $m/\gamma_{s}$. Finally, thanks to how we have chosen $s_{1},\dots,s_{n}$, we may define $f(x_{s_{i}})$ so as to ensure that the relation $\sum_{s}f(x_{s})=0\in \ZZ/{m}$ is satisfied. Let $\ZZ/{m'}\subset \ZZ/{m}$ be the image of $f$.
	
	By construction, there are at least $2N+4rm^{2}-rm^{2}-n\ge 2N+2rm^{2}$ non-trivial orbifold points $s$ such that the restriction of $f$ to $H_{s}$ is injective. If $H_{s}\to \ZZ/{m'}$ is injective, there exists a unique point $s'\in S'$ over $s$ and the ramification index at $s'$ of the corresponding map of coarse moduli spaces is $|H_{s}|-1\ge 1$. By Riemann-Hurwitz,
	\[g_{S'}> m'(g_{S}-1)+N+rm^{2}\ge N+rm^{2}-m'\ge N.\]
\end{proof}

Clearly $D_{2^{\infty}}$ is the $2$-adic completion of the infinite dihedral group $D_{\infty}=\ZZ\rtimes \ZZ/{2}$, and there is a section $D_{2^{\infty}}\subset \widehat{D_{\infty}}=\hat{\ZZ}\rtimes \ZZ/{2}$ of $\widehat{D_{\infty}}\to D_{2^{\infty}}$. Moreover, $D_{2^{\infty}}^{\rm ab}\simeq \left(\ZZ/{2}\right)^{2}$ and $[D_{2^{\infty}},D_{2^{\infty}}]\simeq\ZZ_{2}$.

The following fact, along with its proof, was kindly told to us by J. Stix.

\begin{lemma}\label{lemma:dihedral}
	The free product $\ZZ/{2}\ast\ZZ/{2}$ is isomorphic to the infinite dihedral group $D_{\infty}$.
\end{lemma}

\begin{proof}
	Write $D_{\infty}=\left<r,s\mid s^{2}=(rs)^{2}=1\right>$. Consider the action of $D_{\infty}$ on $\RR$ where $r$ acts as $x\mapsto x+2$ and $s$ acts as $x\mapsto -x$, the statement then follows from \cite[\S I.4, Theorem 6]{serre-trees}. Alternatively, if $\ZZ/{2}\ast\ZZ/{2}=\left<a,b\mid a^{2}=b^{2}=1\right>$, the homomorphisms $a\mapsto s$, $b\mapsto rs$ and $r\mapsto ba$, $s\mapsto a$ are inverses.
\end{proof}

It follows that $D_{2^{\infty}}$ is the $2$-adic completion of the fundamental group $\ZZ/{2}\ast\ZZ/{2}$ of the disk with two orbifold points of order $2$.

\begin{proof}[Proof of Proposition~\ref{proposition:subcyclic}]
	If $D_{2^{\infty}}\subset\hat{\pi}$, then up to passing to an étale cover of $S$ we may assume that $\hat{\pi}^{\rm ab}\supset \left(\ZZ/{2}\right)^{2}=D_{2^{\infty}}^{\rm ab}$ has non-trivial torsion, hence $S$ has non-trivial orbifold points.
	
	Let us now prove the main statement. We break down the proof in two steps. In the first step, we prove it under the assumption that $[G,G]$ is pro-cyclic. In the second step, we prove that $[G,G]$ cannot be isomorphic to $D_{2^{\infty}}$.
	
	By induction on the derived length, these two facts are sufficient. In fact, if the derived length of $G$ is $1$ then $G$ is abelian and $[G,G]$ is trivial, so $G$ is pro-cyclic by the first step. If the derived length is $\ge 2$, then $[G,G]$ is pro-cyclic thanks to the induction hypothesis and the second step; it follows that $G$ has derived length $2$ and it is isomorphic to $D_{2^{\infty}}$ by the first step.

	{\bf Step 1.} Assume that $[G,G]$ is pro-cyclic, we want to show that $G$ is either pro-cyclic or isomorphic to $D_{2^{\infty}}$. If $G^{\rm ab}$ has $p$-rank $1$, then $G$ is pro-cyclic by Lemma~\ref{lemma:psub}. Assume that $G^{\rm ab}$ has $p$-rank $\ge 2$, we want to show that $G\simeq D_{2^{\infty}}$.
	
	It is enough to prove this in the case in which $G^{\rm ab}$ has finite $p$-rank: if $G$ has infinite $p$-rank, since $[G,G]$ has finite $p$-rank we may choose a closed subgroup $G'\subset G$ such that $G'^{\rm ab}$ has finite $p$-rank larger than $3>2=\operatorname{rk} D_{2^{\infty}}^{\rm ab}$, giving a contradiction. Hence, we may assume that $G^{\rm ab}$ has finite $p$-rank $\ge 2$. Up to passing to a finite index subgroup of $\pi$ we might also assume that the image of $G$ in $\hat{\pi}^{\rm ab}/p$ has rank equal to $\operatorname{rk}G^{\rm ab}\ge 2$.

	Up to passing to another finite index subgroup of $\pi$, thanks to Lemma~\ref{lemma:genus} we may assume that the genus $g$ of $S$ is $\ge 2$. Consider the usual presentation
	\[\pi=\left<a_{1},b_{1},\dots,a_{g},b_{g},c_{1},\dots,c_{n}\mid c_{i}^{\gamma_{i}}=1,~\prod [a_{i},b_{i}]\cdot\prod c_{i}=1 \right>.\]
	
	{\bf Case $p\neq 2$.} Let us check that there exists a quotient $\pi\to\pi'$ of the form
	\[\pi'=\left<x_{1},x_{2},x_{3}\mid x_{1}^{p}=x_{2}^{p}=x_{3}^{p}=x_{1}x_{2}x_{3}=1\right>\]
	such that $G\to\pi'^{\rm ab}/p\simeq \left(\ZZ/{p}\right)^{2}$ is surjective.
	
	We have quotients
	\[\tau=\left<a_{1},b_{1},\dots,a_{g},b_{g}\mid\prod [a_{i},b_{i}]=1 \right>,\]
	\[\tau'=\left<c_{1},\dots,c_{n}\mid c_{i}^{\gamma_{i}}=1,~\prod c_{i}=1 \right>\]
	of $\pi$, and an induced surjective map $\pi\to\tau\ast\tau'$ whose abelianization is an isomorphism. If the image of $G$ in $\tau^{\rm ab}/p$ has rank $\ge 2$, we may define the desired quotient $\tau\to\pi'$ using Lemma~\ref{lemma:free}, since $\pi'$ is a quotient of $\ZZ\ast \ZZ$. If the image of $G$ in $\tau'^{\rm ab}/p$ has rank $\ge 2$, then we may define the desired quotient $\tau'\to\pi'$ by mapping $3$ suitably chosen generators of $\tau'$ to $x_{1},x_{2},x_{3}$, and the others to the identity. The only remaining case is the one in which the images of $G$ in $\tau^{\rm ab}/p$ and $\tau'^{\rm ab}/p$ have both rank $1$. Analogously to the above, we may find quotients $\tau\to\ZZ/{p}$, $\tau'\to\ZZ/{p}$ such that the composition $G\to\tau\ast\tau'\to\ZZ/p\ast\ZZ/p\to \left(\ZZ/{p}\right)^{2}$ is surjective, and conclude since $\pi'$ is a quotient of $\ZZ/{p}\ast \ZZ/{p}$.

	The group $\pi'$ is the fundamental group of a complete orbicurve $S'$ of genus $0$ with three orbifold points of degree $p$. The $\left(\ZZ/{p}\right)^{2}$ Galois covering $S''\to S'$ associated with the abelianization morphism $\pi'\to \left(\ZZ/{p}\right)^{2}$ defines a surface $S''$ of genus $g''$. Because of this, the pro-$p$ completion $Q$ of $\pi'/[[\pi',\pi'],[\pi',\pi']]$ is a pro-$p$-group with $Q^{\rm ab}\simeq \left(\ZZ/{p}\right)^{2}$ and $[Q,Q]=\ZZ_{p}^{2g''}$. By Lemma~\ref{lemma:psub}, $G\to Q$ is surjective. Since $p\ge 3$ then $\chi(S'')=p^{2}\chi(S')\le 0$, hence $g''\ge 1$ and $[Q,Q]$ has rank $\ge 2$, which is in contradiction with the fact that $[G,G]$ is pro-cyclic.
	
	{\bf Case $p=2$, $G^{\rm ab}\not\simeq \left(\ZZ/{2}\right)^{2}$.} If $G^{\rm ab}\simeq \left(\ZZ/{2}\right)^{m}$ for some $m\ge 3$, with an argument analogous to the above we may find a quotient $\pi'$ of $\pi$ of the form
	\[\left<x_{1},x_{2},x_{3},x_{4}\mid \forall i:x_{i}^{2}=1,~x_{1}x_{2}x_{3}x_{4}=1\right>\]
	and such that $G\to\hat{\pi}'^{\rm ab}$ is surjective. Otherwise, using another analogous argument we might choose $\pi'$ of the form
	\[\left<x_{1},x_{2},x_{3}\mid x_{1}^{4}=x_{2}^{4}=x_{3}^{2}=x_{1}x_{2}x_{3}=1\right>.\]
	In both cases, $\pi'$ is the fundamental group of an orbicurve $S'$ with $\chi(S')= 0$. In both cases, the abelianization of $\pi'$ is a finite $2$-group and defines a Galois cover which is a surface, hence we might get a contradiction as in the case $p\ge 3$.
	
	{\bf Case $p=2$, $G^{\rm ab}\simeq \left(\ZZ/{2}\right)^{2}$.} Up to passing to a finite index subgroup of $\pi$ we might assume that $G^{\rm ab}\subset\hat{\pi}^{\rm ab}/2$. In particular, $S$ has at least $3$ orbifold points of even degree. Since the genus is positive and $G^{\rm ab}$ is torsion, the pullback of any abelian, non-trivial étale cover of the coarse moduli space of $S$ defines an étale cover of $S$ whose étale fundamental group contains $G$ and which has at least $6$ orbifold points of even order. Hence, we may reduce to the case in which $S$ has at least $4$ orbifold points of even degree.
	
	If $S$ has at least $4$ orbifold points of even degree and since $G^{\rm ab}\subset \hat{\pi}^{\rm ab}/2$, there exists a quotient $\pi'$ of $\pi$ of the form
	\[\left<x_{1},x_{2},x_{3},x_{4}\mid \forall i:x_{i}^{2}=1,~x_{1}x_{2}x_{3}x_{4}=1\right>\]
	such that the image of $G$ in $\hat{\pi}^{\rm ab}=\pi^{\rm ab}=\left(\ZZ/{2}\right)^{4}/(1,1,1,1)\simeq \left(\ZZ/{2}\right)^{3}$ is isomorphic to $\left(\ZZ/{2}\right)^{2}$. The group $\pi'$ is the fundamental group of an orbicurve $S'$ of genus $0$ with $4$ orbifold points of degree $2$.
	
	Consider the homomorphism $\pi'\to \ZZ/{2}$ which maps $x_{i}$ to $1$ for every $i$, the associated covering $S''\to S'$ is a torus with fundamental group $\pi''=\ZZ^{2}\subset \pi'$. If the image of $G$ is contained in $\hat{\pi}''\subset\hat{\pi}'$, then $G^{\rm ab}\simeq \left(\ZZ/{2}\right)^{2}$ embeds into $\hat{\pi}''=\hat{\ZZ}^{2}$, which is absurd since the latter is torsion free. Because of this and by rank reasons, there exists an element in the kernel of $\left(\ZZ/{2}\right)^{4}/(1,1,1,1)\to \ZZ/{2}$ which is not in the image of $G$, say $(1,1,0,0)$ for simplicity.
	
	Consider the homomorphism $\pi'\to \ZZ/{2}\ast \ZZ/{2}$ which maps $x_{1},x_{2}$ to the first generator and $x_{3},x_{4}$ to the second generator. The kernel of the induced map $\left(\ZZ/{2}\right)^{4}/(1,1,1,1)\to \left(\ZZ/{2}\right)^{2}$ between abelianizations is generated by $(1,1,0,0)$, which is not contained in the image of $G$. Because of this, the composition $G^{\rm ab}=\left(\ZZ/{2}\right)^{2}\hookrightarrow \left(\ZZ/{2}\right)^{4}/(1,1,1,1)\to \left(\ZZ/{2}\right)^{2}$ is bijective, which in turn implies that the composition $G\to\widehat{\ZZ/{2}\ast \ZZ/{2}}\to D_{2^{\infty}}$ is surjective. Since $[G,G]$ is a pro-cyclic pro-$2$-group, the induced surjective homomorphism $[G,G]\to [D_{2^{\infty}},D_{2^{\infty}}]\simeq\ZZ_{2}$ must be injective as well; this implies that $G\to D_{2^{\infty}}$ is injective and hence an isomorphism.
	
	{\bf Step 2.} Finally, we have to check that $[G,G]$ cannot be isomorphic to $D_{2^{\infty}}$. Clearly, we may assume that $p=2$ and that $G$ is not pro-cyclic; in particular, $G^{\rm ab}$ is not cyclic. The construction of the preceeding case gives us a surjective homomorphism $G\to D_{2^{\infty}}$, which induces a surjective homomorphism $[G,G]\to[D_{2^{\infty}},D_{2^{\infty}}]\simeq\ZZ_{2}$. We conclude that $[G,G]\not\simeq D_{2^{\infty}}$ since $D_{2^{\infty}}^{\rm ab}\simeq \left(\ZZ/{2}\right)^{2}$.
\end{proof}

\begin{example}
	The group $D_{2^{\infty}}$ can actually appear as a subgroup of the profinite completion of the fundamental group of an hyperbolic orbicurve. Let $M$ be the unit disk, and consider a complete curve $S$ of arbitrary genus with a \emph{continuous} (not holomorphic) map to the disk $S\to M$ such that for some open subdisk $U\subset M$ the restriction $S_{U}\to U$ is a disjoint union of copies of $U$. Define $M_{2,2}\to M$ as an orbifold whose coarse moduli space is the disk $M$ and which has two orbifold points of order $2$ over points of $U\subset M$, and let $U_{2,2}\subset M_{2,2}$ be the inverse image of $U$. 
	
	Define $S'$ as the fibered product $S\times_{M}M_{2,2}$, then $S'$ is an orbisurface whose coarse moduli space is $S$ with maps $S'\to M_{2,2}$, $U_{2,2}\to S'$ such that the composition $U_{2,2}\to M_{2,2}$ is the given embedding $U_{2,2}\subset M_{2,2}$ and thus defines an isomorphism $\pi_{1}(U_{2,2})\simeq \pi_{1}(M_{2,2})$. Since $D_{2^{\infty}}\subset\widehat{D_{\infty}}=\widehat{\pi_{1}(M_{2,2})}$, then $\widehat{\pi_{1}(S')}$ has a closed subgroup isomorphic to $D_{2^{\infty}}$.
\end{example}

\section{Proof of Theorem~\ref{theorem:index}}\label{section:index}

Let us now prove Theorem~\ref{theorem:index}. Denote by $S$ the quotient stack $[C/G]$. Up to passing to a finite extension of $k$, we might assume that $S(k)$ is non-empty.

\subsection{Definition of $\fC$} Recall that, by assumption, all curves are smooth, projective and geometrically connected. Let $P=\upic^{1}_{C}$ be the Picard scheme of line bundles of degree $1$ on $C$, the action of $G$ on $C$ induces an action on $P$.

Denote by $X$ the quotient stack $[P/G]$. Since the action on $C\subset P$ is faithful then the action on $P$ is generically free, hence $X$ is generically a scheme. Let $K=k(X)$ be its function field, then $k(P)/K$ is a Galois extension with Galois group $G$ and $\spec k(P)\to\spec K$ is a $G$-torsor associated with a morphism $\spec K\to\cB G$. As in \S\ref{sect:obstruction}, define $\fC$ as the twist of $C_{K}$ by this torsor, namely
\[\fC\eqdef C_{K}\times^{G}\spec k(P)=(C_{K}\times_{K}\spec k(P))/G.\]

\subsection{The index of $\fC$}

We want to show that $\fC$ has index $1$. By construction, $\upic^{1}_{\fC}=P_{K}\times^{G}\spec k(P)$ has a $K$-rational point and hence $\fC$ has period $1$. Let $I$ be the index of $\fC$. By Lichtenbaum's theorem \cite[Theorem 8]{lichtenbaum}, we have either $I=1$ or $I=2$, and if $I=2$ then $g_{C}-1$ is odd, where $g_{C}$ is the genus of $C$. Since $S(k)\neq\emptyset$, the index divides $|G|$, hence if $p\neq 2$ then $I=1$.

Assume $p=2$, it is sufficient to show that $g_{C}-1$ is even. If by contradiction $g_{C}-1$ is odd, by the assumption in Theorem~\ref{theorem:index} either the action is free or $G$ has no cyclic subgroups of index $2$. In both cases $4$ divides both $|G|$ and the cardinality of each orbit (since geometric stabilizers are cyclic), hence by Riemann-Hurwitz $4$ divides $2g_{C}-2$ as well, which is absurd.

\subsection{Galois sections of $\fC$}

It remains to prove that $\Pi_{\fC/K}(K)=\emptyset$. By Proposition~\ref{proposition:obstruction}, it is enough to show that $T=\spec k(P)\to\spec K$ does not lift to a $\underline{\pi}_{1}(S,p)$-torsor, where $p\in S(k)$ is some rational point. Equivalently, we want to show that $\spec K\to\cB G$ does not lift to $\Pi_{S/k}=\cB\underline{\pi}_{1}(S,p)$. 

Assume by contradiction that there exists a section $\spec K\to\Pi_{S/k}$ which lifts $\spec K\to\cB G$. We would like to show that $\spec K\to\Pi_{S/k}$ extends to a morphism $X\to\Pi_{S/k}$; since $X=[P/G]$ is smooth, it is well-known how to do this using a weight argument. References are only available for schemes, though, while $X$ and $S$ are stacks; it is possible to sidestep this problem by working with $P$ and $C$ instead of $X$ and $S$.

We have a $2$-cartesian diagram
\[\begin{tikzcd}
	\spec k(P)\rar\dar		&	\Pi_{C/k}\rar\dar	&	\spec k\dar	\\
	\spec K\rar				&	\Pi_{S/k}\rar		&	\cB G.
\end{tikzcd}\]

By a well-known weight argument (see e.g. \cite[Corollary A.11]{giulio-proed}, \cite[Lemma 3.3]{giulio-scfg}) $\spec k(P)\to\Pi_{C/k}$ extends to a morphism $P\to\Pi_{C/k}$. By descent theory, this induces a morphism $X\to\Pi_{S/k}$. It is possible to avoid using descent theory as follows.

Let $R\subset\gal\left(\bar{K}/k(P)\right)$ be the kernel of $\gal\left(\bar{K}/k(P)\right)\to\pi_{1}(P)$. We have a commutative diagram of short exact sequences
\[\begin{tikzcd}
	1\rar	&	\gal\left(\bar{K}/k(P)\right)\rar\dar	&	\gal\left(\bar{K}/k(X)\right)\rar\dar	&	G\rar\ar[d,equal]	&	1	\\
	1\rar	&	\pi_{1}(P)\rar							&	\pi_{1}(X)\rar							&	G\rar				&	1,
\end{tikzcd}\]
hence $\pi_{1}(X)=\gal\left(\bar{K}/k(X)\right)/R$. Since $\spec k(P)\to\Pi_{S/k}$ extends to $P$, the associated homomorphism $\gal\left(\bar{K}/k(P)\right)\to\pi_{1}(S)$ maps $R$ to the identity and hence we get a factorization
\[\gal\left(\bar{K}/k(X)\right)\to\pi_{1}(X)\to\pi_{1}(S)\]
of the homomorphism $\gal\left(\bar{K}/k(X)\right)\to\pi_{1}(S)$ associated with the section $\spec K\to\Pi_{S/k}$. 

Since $\spec K\to\Pi_{S/k}$ lifts $\spec K\to\cB G$, the composition
\[\pi_{1}(X_{\bar{k}})\to\pi_{1}(S_{\bar{k}})\to G\]
is a factorization of the surjective homomorphism $\pi_{1}(X_{\bar{k}})\to G$ induced by the $G$-covering $P\to X$.

Since $\pi_{1}(P_{\bar{k}})$ is abelian and $G$ is a finite $p$-group, the image $H\subset \pi_{1}(S_{\bar{k}})$ of a $p$-Sylow of $\pi_{1}(X_{\bar{k}})$ is an extension of $G$ by an abelian pro-$p$-group; in particular $H$ has finite derived length. By Proposition~\ref{proposition:subcyclic}, such a subgroup of $\pi_{1}(S_{\bar{k}})$ is either pro-cyclic or isomorphic to $D_{2^{\infty}}$. If $H$ is pro-cyclic, then $G$ is cyclic, giving a contradiction. If $H\simeq D_{2^{\infty}}$, then $p=2$ and $G$ is dihedral. By hypothesis this implies that the action is free, which is contradiction with Proposition~\ref{proposition:subcyclic}.

\begin{remark}
	If the section conjecture holds for $C_{k(P)}$ and the action is free, we only need to assume that $G$ is non-trivial to obtain that the curve $\fC$ constructed above satisfies $\Pi_{\fC/K}(K)=\emptyset$. In fact, if $\spec K\to\cB G$ lifts to $\Pi_{S/k}$, the induced morphism $\spec k(P)\to \Pi_{C/k}$ is associated with a point $\spec k(P)\to C$ which factorizes through $\spec k$ since $C$ is hyperbolic and $P$ is an abelian variety. With notation as above, this implies that $H\simeq G$ is finite, which is absurd since $\pi_{1}(S_{\bar{k}})$ is torsion free.
	
	On the other hand, this is false if the action is not free and $G$ is cyclic. For instance, if $G$ is cyclic and fixes a rational point $c\in C(k)$, then $c$ induces a $K$-rational point of $\fC$.
\end{remark}

\section{Proof of Theorem~\ref{theorem:cover}}

Consider the action $\phi$ of $\ZZ/{\ell}$ on $\ZZ/{\ell^{2}}$ given by $1\mapsto 1+\ell$, there exists a surjective homomorphism $\ZZ/{\ell^{2}}\rtimes_{\phi} \ZZ/{\ell}\to \left(\ZZ/{\ell}\right)^{2}$. 

\begin{lemma}\label{lemma:semi}
	Let $k,\ell$ be as in Theorem~\ref{theorem:cover}. There exists a finite extension $K/k$ with a surjective homomorphism $\gal\left(\bar{K}/K\right)\to \left(\ZZ/{\ell}\right)^{2}$ which does not lift to $\ZZ/{\ell^{2}}\rtimes_{\phi} \ZZ/{\ell}$. 
\end{lemma}

\begin{proof} 
	Up to a finite extension, we may assume that $k$ contains a primitive $\ell^{2}$-th root of $1$.
	
	Assume first that $k$ is a finite extension of $\QQ_{p}$, $p\neq\ell$. Since $p\neq\ell$, the maximal $\ell$-adic quotient of $\gal(\bar{k}/k)$ is $\ZZ_{\ell}\rtimes_{\rho}\ZZ_{\ell}$, where $\ZZ_{\ell}$ acts on itself with the cyclotomic character $\rho$. Since $k$ contains an $\ell^{2}$-th root of $1$, the action of $\ZZ_{\ell}$ on the quotient $\ZZ/{\ell^{2}}$ of $\ZZ_{\ell}$ is trivial, hence there exists a surjective homomorphism $\ZZ_{\ell}\rtimes_{\rho}\ZZ_{\ell}\to \left(\ZZ/{\ell^{2}}\right)^{2}$; furthermore, every finite quotient of $\ZZ_{\ell}\rtimes_{\rho}\ZZ_{\ell}$ of exponent $\ell^{2}$ is a quotient of $\left(\ZZ/{\ell^{2}}\right)^{2}$.
	
	It follows that there are no surjective homomorphisms $\ZZ_{\ell}\rtimes_{\rho}\ZZ_{\ell}\to \ZZ/{\ell^{2}}\rtimes_{\phi} \ZZ/{\ell}$. By Lemma~\ref{lemma:psub}, this implies that there are no liftings $\ZZ_{\ell}\rtimes_{\rho}\ZZ_{\ell}\to \ZZ/{\ell^{2}}\rtimes_{\phi} \ZZ/{\ell}$ of the composition $\ZZ_{\ell}\rtimes_{\rho}\ZZ_{\ell}\to \left(\ZZ/{\ell^{2}}\right)^{2}\to \left(\ZZ/{\ell}\right)^{2}$. This concludes the proof for local fields.
	
	If $k$ is finitely generated over $\QQ$, let $h\subset k$ be the algebraic closure of $\QQ$ in $k$. I claim that there exists a finite extension $h'/h$ with a section $\gal(\bar{h}/h')\to\gal(\bar{k}/k)$. By induction, it is enough to do the case in which $k$ has transcendence degree $1$ over $h$. Let $V$ be a smooth curve over $h$ whose fraction field is $k$, and $v\in V$ a closed point with residue field $h'$. Let $k_{v}$ be the completion of $k$ at the place defined by $v$, the residue field of $k_{v}$ is $h'$ and hence we get the desired splitting $\gal(\bar{h}/h')\to\gal(\bar{k}_{v}/k_{v})\to\gal(\bar{k}/k)$. Up to replacing $h$ with $h'$ and $k$ with $h'\otimes_{h}k$, we might then assume that $\gal(\bar{k}/k)\to\gal(\bar{h}/h)$ is split. 
	
	Let $h_{\nu}$ be the completion of $h$ at a prime $\nu$ dividing $p\neq \ell$. Since $\gal(\bar{k}/k)\to\gal(\bar{h}/h)$ is split and $\gal(\bar{h}_{\nu}/h_{\nu})\to\gal(\bar{h}/h)$ is injective, then $\gal(\bar{k}/k)$ has a closed subgroup isomorphic to $\gal(\bar{h}_{\nu}/h_{\nu})$. Let $\gal(\bar{h}_{\nu}/h_{\nu})\to \left(\ZZ/{\ell}\right)^{2}$ be the only surjective homomorphism to $\left(\ZZ/{\ell}\right)^{2}$, since $\gal(\bar{h}_{\nu}/h_{\nu})$ is a closed subgroup of $\gal(\bar{k}/k)$ there exists a finite index subgroup $\Gamma\subset\gal(\bar{k}/k)$ containing $\gal(\bar{h}_{\nu}/h_{\nu})$ with a factorization $\gal(\bar{h}_{\nu}/h_{\nu})\to \Gamma\to \left(\ZZ/{\ell}\right)^{2}$. Since $k$, and hence $h_{\nu}$, contains a primitive $\ell^{2}$-rooth of $1$, then $\Gamma\to(\ZZ/\ell)^{2}$ does not lift to $\ZZ/{\ell^{2}}\rtimes_{\phi} \ZZ/{\ell}$ and we may choose $K/k$ as the field fixed by $\Gamma$.
\end{proof}

Let us now prove Theorem~\ref{theorem:cover}. Since $S$ is hyperbolic, up to a finite extension of $k$ by Lemma~\ref{lemma:free} we may assume that there exists a geometrically connected, finite étale Galois covering $U\to S$ with Galois group $\ZZ/{\ell^{2}}\rtimes_{\phi} \ZZ/{\ell}$; denote by $C\to S$ the induced $\left(\ZZ/\ell\right)^{2}$-cover. 

Let $K/k$, $\gal(\bar{K}/K)\to\left(\ZZ/\ell\right)^{2}$ be as in Lemma~\ref{lemma:semi}, the homomorphism corresponds to a $\left(\ZZ/\ell\right)^{2}$-torsor $T\to \spec K$. The twist $\fC=C_{K}\times^{\left(\ZZ/\ell\right)^{2}} T\to S_{K}$ is a torsor for the group scheme $\underline{\aut}(T)$ of $\left(\ZZ/\ell\right)^{2}$-equivariant automorphisms of $T$. Since $\left(\ZZ/\ell\right)^{2}$ is abelian, then $\underline{\aut}(T)\simeq \left(\ZZ/\ell\right)^{2}$ and hence $\fC\to S_{K}$ is a Galois $\left(\ZZ/\ell\right)^{2}$-cover. By Proposition~\ref{proposition:obstruction} the twist $\fC$ satisfies the section conjecture. 

\bibliographystyle{amsalpha}
\bibliography{index}

\end{document}